\documentclass[12pt]{amsart}
\usepackage{amsmath}
\usepackage{amsfonts}
\usepackage{hyperref}
\usepackage{amssymb}
\usepackage{latexsym}
\usepackage{amsthm}
\usepackage{a4wide}

\usepackage{graphicx}

\usepackage[ansinew]{inputenc}
\usepackage{amsfonts,epsfig}
\usepackage{mathabx}

\newtheorem{theorem}{Theorem}

\newtheorem{coro}[theorem]{Corollary}

\newtheorem{proposition}[theorem]{Proposition}
\newtheorem{rk}[theorem]{Remark}
\newcounter{other}            

\theoremstyle{definition}
\newtheorem{definition}[theorem]{Definition}

\numberwithin{equation}{section}
\begin{document}

\title[General fractional derivatives]{General fractional derivatives and the Bergman projection}

\author[Antti Per\"{a}l\"{a}]{Antti Per\"{a}l\"{a}}
\address{Antti Per\"{a}l\"{a} \\Departament de Matem\`{a}tiques i Inform\`{a}tica \\
Universitat de Barcelona\\
08007 Barcelona\\
Catalonia, Spain\\
Barcelona Graduate School of Mathematics (BGSMath).} \email{perala@ub.edu}


%
\subjclass[2010]{30H10, 32A36, 46A20, 42B35}

\keywords{Bergman space, Besov space, Bergman projection,  doubling weight, fractional derivative}

\thanks{The author acknowledges financial support from the Spanish Ministry of Economy and Competitiveness, through the Mar\'ia de Maeztu Programme for Units of Excellence in R\&D (MDM-2014-0445). The author was partially supported by the grant MTM2017-83499-P (Ministerio de Educaci\'on y Ciencia).}


\begin{abstract}
In this note we study some basic properties of general fractional derivatives induced by weighted Bergman kernels. As an application we demonstrate a method for generating pre-images of analytic functions under weighted Bergman projections. This approach is useful for proving the surjectivity of weighted Bergman projections in cases when the target space is not a subspace of the domain space (such situations arise often when dealing with Bloch and Besov spaces). We also discuss a fractional Littlewood-Paley formula.
\end{abstract}

\maketitle



\section{Introduction}

\noindent According to the classical Bergman reproducing formula, for a sufficiently nice analytic function $f$ defined on the unit disk $\mathbb{D}=\{|z|<1\}$, we have
$$f(z)=(\alpha+1)\int_{\mathbb{D}}\frac{f(\xi)(1-|\xi|^2)^\alpha dA(\xi)}{(1-z\overline{\xi})^{2+\alpha}},$$
where $\alpha>-1$ and $dA(\xi)=\pi^{-1}dx dy$ (for $\xi = x+iy$) is the normalized Lebesgue area measure. Differentiating this identity $N$ times gives us
$$f^{(N)}(z)=C(\alpha,N)\int_{\mathbb{D}}\frac{f(\xi)\overline{\xi}^N (1-|\xi|^2)^\alpha dA(\xi)}{(1-z\overline{\xi})^{2+\alpha+N}},$$
where $C(\alpha,N)$ is the corresponding normalizing constant. It is clear that for many purposes, the properties of $f^{(N)}$ are the same as the properties of
$$R^{\alpha,N}f(z)=\int_{\mathbb{D}}\frac{f(\xi) (1-|\xi|^2)^\alpha dA(\xi)}{(1-z\overline{\xi})^{2+\alpha+N}},$$
but getting rid of the factor $\overline{\xi}^N$ allows us to understand this formula even when $N$ is not an integer. These observations lead us to the fractional derivatives that were studied by Kehe Zhu in \cite{ZhuSmall}. See also \cite{ZZ} and \cite{ZhuBn}.

In this paper, we will study a related, more general, concept of fractional derivatives. Roughly the idea is that given two radial weights $\omega$ and $\nu$, there exists a unique mapping $R^{\omega,\nu}$ defined for all analytic functions, and transforming the Bergman kernel with respect to $\omega$ into the Bergman kernel with respect to $\nu$. In many ways, such operators possess the nice properties of the fractional derivatives defined by Zhu.

As an application the newly introduced concept, we show that it can be used to produce pre-images under the weighted Bergman projection $P_\omega$. This can be useful especially, when one wants to prove surjectivity of $P_\omega:X\to Y$, where $Y$ is not a subspace of $X$. In particular, this works when $X=L^\infty$ and $Y$ is the Bloch space, and $\omega$ satisfies a two-sided doubling condition. This surjectivity result has been obtained by Pel\'aez and R\"atty\"a in \cite{PR2018} by using a different method. We note that the present technique can be also used to prove a similar result when $X=L^p_{\lambda_\omega}$ and $Y$ is $p$-Besov space (see the corresponding section for definitions). In fact, for many uses it suffices to consider $\omega$ and $\omega_\alpha(z)=\omega(z)(1-|z|)^\alpha$, in which case
$R^{\omega,\omega_\alpha}$ can be understood as a fractional derivative of order $\alpha$. We hope and expect that over time these fractional derivatives will find other applications as well.\\

The paper is organized as follows. In Section 2, we introduce the framework of weight classes which we will use. Section 3 contains the definition and basic discussion on the fractional derivatives. Section 4 contains the proof of surjectivity for the Bloch case, and in Section 5, we will deal with the Besov case. Section 6 is the last part of the paper, and it contains some further remarks, most notably a fractional generalization of the classical Littlewood-Paley formula.

Throughout the paper, we write $a\lesssim b$ to indicate that there exists a constant $C>0$ with $a\leq Cb$. The relation $a\gtrsim b$ is defined in an analogous manner. Finally, if $a\lesssim b$ and $a \gtrsim b$, we will write $a\asymp b$.

\section{Weighted Bergman spaces}

A non-negative integrable function $\omega$ on the unit disk $\mathbb{D}$ is called a weight. We will be mostly interested in radial weights: $\omega(z)=\omega(|z|)$ for all $z \in \mathbb{D}$. For a radial weight $\omega$ we define
$$\widehat{\omega}(r)=\int_r^1 \omega(s)ds,\quad r \in [0,1)$$
and we assume that $\widehat{\omega}$ is non-zero for $0\leq r<1$, for otherwise much of the theory to follow would lead to trivialities. For convenience, we also agree that for $z \in \mathbb{D}$, we have $\widehat{\omega}(z):=\widehat{\omega}(|z|)$. We will focus mainly on the weight classes defined below.

\begin{definition}
A radial weight $\omega$ belongs to the class $\widehat{\mathcal{D}}$ if there exists $C_\omega>0$ so that
\begin{equation}\label{above}
\widehat{\omega}(r)\leq C_\omega \widehat{\omega}\left(\frac{1+r}{2}\right).
\end{equation}
Furthermore, $\omega$ belongs to $\widecheck{\mathcal{D}}$, if there exist $K_\omega>1$ and $C'_\omega>1$ so that
\begin{equation}\label{below}
\widehat{\omega}(r)\geq C'_\omega \widehat{\omega}\left(1-\frac{1-r}{K_\omega}\right).
\end{equation}
A radial weight $\omega$ belonging to both $\widehat{\mathcal{D}}$ and $\widecheck{\mathcal{D}}$ is called (two-sided) doubling, and we write $\omega \in \mathcal{D}$.
\end{definition}

More properties of the class $\widehat{\mathcal{D}}$ can be found in \cite{PR2016/1}. The class $\widecheck{\mathcal{D}}$ was introduced in the, at the moment unpublished paper \cite{PR2018}, but the results that are necessary for the present work can be found in \cite{PPR}.

We note that $\mathcal{D}$ contains the class $\mathcal{R}$ of regular weights. We say that a radial weight $\omega$ belongs to $\mathcal{R}$ if $$\frac{\widehat{\omega}(r)}{1-r}\asymp \omega(r).$$
Note that regular weights are not allowed to have zeroes, as is evident from the definition. The weights in $\mathcal{D}$ can have zeroes, but their zero sets cannot be hyperbolically very large near the boundary, see \cite{PPR}.\\

We denote by $\mathcal{H}$ the space of all analytic functions on $\mathbb{D}$. It is equipped with the topology of uniform convergence on compact sets. The weighted $L^p_\omega$ quasi-norm is given by
$$\|f\|_{p,\omega}=\left(\int_{\mathbb{D}}|f(\xi)|^p \omega(\xi)dA(\xi)\right)^{1/p},$$
The weighted Bergman space $A^p_\omega$ consists of analytic functions in $L^p_\omega$. If $\omega \in \widehat{\mathcal{D}}$, the space $A^p_\omega$ is also a quasi-Banach space. Moreover, $A^2_\omega$ is a Hilbert space with inner product
$$\langle f,g\rangle_{\omega}=\int_{\mathbb{D}}f(\xi)\overline{g(\xi)}\omega(\xi)dA(\xi).$$ Recall that a doubling weight $\omega$ induces a reproducing Bergman kernel $B_z^\omega \in A^2_\omega$ by the formula
$$f(z)=\int_{\mathbb{D}}f(\xi)\overline{B_z^\omega(\xi)}\omega(\xi)dA(\xi),\quad f \in A^2_\omega,$$
and the Bergman projection $P_\omega : L^2_\omega\to A^2_\omega$ is the integral operator induced by $B_z^\omega$;
$$P_\omega f(z)=\int_{\mathbb{D}}f(\xi)\overline{B_z^\omega(\xi)}\omega(\xi)dA(\xi),\quad f   \in L^2_\omega.$$

We remark that for $\omega \in \widehat{\mathcal{D}}$, the Bergman kernels $B_z^\omega$ are actually bounded analytic functions for every $z \in \mathbb{D}$. This is evident from the power series expansion
$$B_z^\omega(\xi)=\sum_{k=0}^\infty \frac{1}{\omega_k}(\overline{z}\xi)^k,$$
where $$\omega_k=2\int_0^1 r^{2k+1}\omega(r)dr$$ denotes the $2k+1$ moment of $\omega$. Therefore, the two integral formulas listed above make sense also under $L^1_\omega$ integrability assumption on $f$.

\section{General fractional derivatives}

Following the blueprint of Kehe Zhu \cite{ZhuSmall} (see also \cite{ZZ, ZhuBn}), we now define the general fractional differential operators on the disk. 

\begin{proposition}\label{basic}
Let $\omega$ and $\nu$ be doubling weights. There exists a unique linear operator $R^{\omega,\nu}:\mathcal{H}\to \mathcal{H}$ with the following three properties.
\begin{enumerate}
\item $R^{\omega,\nu}:H\to H$ is continuous;
\item $(R^{\omega,\nu}f)_r=R^{\omega,\nu}f_r$ for every $r\in (0,1)$;
\item $R^{\omega,\nu}B_z^\omega(\xi)=B_z^\nu(\xi)$.
\end{enumerate}
\end{proposition}

\begin{rk}
We remark that the present notation is not completely analogous with the notation $R^{s,t}$ in \cite{ZZ, ZhuBn}. If $\omega(z)=(\alpha+1)(1-|z|^2)^\alpha$ and $\nu(z)=(\alpha+t+1)(1-|z|^2)^{\alpha+t}$, then the operator $R^{\alpha,t}$ in the notation of \cite{ZZ, ZhuBn} is the operator $R^{\omega,\nu}$ in the present notation.
\end{rk}

\begin{proof}
The proof is the same as in \cite{ZhuSmall} apart from some minor modifications to the setting of general weights. We present the proof for the convenience of the reader.\\

\noindent {\bf Uniqueness:} By the reproducing formula, if $f \in \mathcal{H}$ and $r\in (0,1)$, we have
$$f_r(z)=\int_{\mathbb{D}}f_r(\xi)\overline{B_z^\omega(\xi)}\omega(\xi)dA(\xi).$$
From this, we obtain
\begin{equation}\label{Rr}
R^{\omega,\nu}f_r(z)=\int_{\mathbb{D}}f_r(\xi)\overline{B_z^\nu(\xi)}\omega(\xi)dA(\xi).
\end{equation}

Since this integral can be approximated by finite sums of the kernel functions, uniformly on compact sets, this clearly implies the uniqueness.\\

\noindent {\bf Existence:} For a function $f(z)=\sum f_k z^k$, which is analytic on a larger disk, we can obviously define
\begin{equation}\label{sum}
R^{\omega,\nu}f(z)=\int_{\mathbb{D}}f(\xi)\overline{B_z^\nu(\xi)}\omega(\xi)dA(\xi)=\sum_{k=0}^\infty \left(\frac{\omega_k}{\nu_k}\right)f_k z^k.
\end{equation}
From this it is easy to see that for such $f$, the condition (2) is satisfied. Also, by the reproducing formula, the condition (3) holds. For a general $f \in \mathcal{H}$ and $z \in \mathbb{D}$, we note that there exists $r \in (0,1)$ and $z' \in \mathbb{D}$ with $r=rz'$. We may then define
$$R^{\omega,\nu}f(z):=R^{\omega,\nu}f_r(z').$$
We remark that if $z=r_1z_1=r_2z_2$ with $r_1=rr_2$, ($z_1,z_2 \in \mathbb{D}$, $0<r_1<r_2<1$ and $r \in (0,1)$), then
$$R^{\omega,\nu}f_{r_1}(z_1)=R^{\omega,\nu}f_{rr_2}(z_1)=R^{\omega,\nu}f_{r_2}(rz_1)=R^{\omega,\nu}f_{r_2}(z_2),$$
since $rz_1=z_2$. We see that $R^{\omega,\nu}:\mathcal{H}\to \mathcal{H}$ is well-defined and satisfies (2). Also, it is clear that $R^{\omega,\nu}$ is linear; to see that $R^{\omega,\nu}f$ is analytic, just notice that any dilatation $(R^{\omega,\nu}f)_r$ clearly is.

We finally show the continuity. To this end, suppose that $f_k\to f$ in $\mathcal{H}$ (uniformly on compact sets). Then, for the dilatations $f_k(rz)\to f(rz)$ uniformly on $\mathbb{D}$, where $r \in (0,1)$. Using formula \eqref{sum} and (2), we see that
\begin{align*}
R^{\omega,\nu}f_k(rz)&=\int_{\mathbb{D}}f_k(\sqrt r\xi)\overline{B_{\sqrt r z}^\nu (\xi)}\omega(\xi)dA(\xi)\\
&\to \int_{\mathbb{D}}f(\sqrt r\xi)\overline{B_{\sqrt r z}^\nu (\xi)}\omega(\xi)dA(\xi)\\
&=R^{\omega,\nu}f(rz)
\end{align*}
uniformly on $z \in \mathbb{D}$. We get $R^{\omega,\nu}f_k\to R^{\omega,\nu}f$ uniformly on compact sets, since $r \in (0,1)$ is arbitrary. This finishes the proof of existence, and we are done.
\end{proof}

We remark that $R^{\omega,\nu}$ is actually a Taylor coefficient multiplier, mapping
$f(z)=\sum f_kz^k$ to
$$R^{\omega,\nu}f(z)=\sum_{k=0}^\infty \left(\frac{\omega_k}{\nu_k}\right)f_k z^k.$$
It can be shown that if $f$ is analytic on $\mathbb{D}$, then so is $R^{\omega,\nu}f$. An alternative method to prove Proposition \ref{basic} can be found in Proposition 1.14 of \cite{ZhuBn}.

By looking at the discussion in \cite{ZhuSmall}, it seems possible that this idea of fractional derivatives can be extended to much more general classes of even non-radial weights. However, the weights being radial guarantee that the kernel functions are bounded analytic functions. This makes it possible to adapt many results from the more standard cases with little analysis. Fractional derivatives induced by non-radial weights seem an intriguing topic, but they are not considered in this paper.

Note that if $f$ is an analytic function defined in terms on a finite Borel measure $\mu$ as
$$f(z)=\int_{\mathbb{D}}\overline{B_z^\omega(\xi)}d\mu(\xi),$$
then we have
$$R^{\omega,\nu}f(z)=\int_{\mathbb{D}}\overline{B_z^\nu(\xi)}d\mu(\xi).$$
In particular, if $f \in A^1_\omega$, then
\begin{equation}\label{int}
R^{\omega,\nu}f(z)=\int_{\mathbb{D}}f(\xi)\overline{B_z^\nu(\xi)}\omega(\xi)dA(\xi),
\end{equation}
or simply $R^{\omega,\nu}f(z)=\langle f,B_z^\nu\rangle_\omega$. When $f \notin A^1_\omega$ a nice integral formula can still often be obtained. For this purpose, let $\omega$ be a radial, integrable weight and set
$$\omega_{+}(r)=2\int_r^1 \omega(s)\frac{ds}{s}.$$
Further, we understand $\omega=\omega_{+0}$ and set $\omega_{+n}=(\omega_{+(n-1)})_+$. An easy calculation using Fubini's theorem reveals us 
\begin{equation}\label{plus}
(\omega_+)_n=4\int_0^1 r^{2n+1}\int_r^1 \frac{\omega(s)}{s}dsdr=4\int_0^1 \frac{\omega(s)}{s}\int_0^s r^{2n+1} drds=\frac{\omega_n}{n+1}.
\end{equation}

We can now present an analog of Proposition 5 in \cite{ZZ}. In a way, the following result is even stronger, as there is no polynomial "error" factor present. We remark that since the spaces covered in the present paper are very small in nature, using the following proposition will not be necessary. Nevertheless, it might be useful in future works.

\begin{proposition}
Let $\omega,\nu \in \widehat{\mathcal{D}}$. Then
$$R^{\omega,\nu}B_z^{\omega_{+N}}(\xi)=B_z^{\nu_{+N}},$$
for every $N \in \mathbb{N}$. Moreover, we have $R^{\omega,\nu}=R^{\omega_{+N},\nu_{+N}}$ for every $N$.
\end{proposition}

\begin{proof}
Since both $R^{\omega,\nu}$ and $R^{\omega,\omega_{+N}}$ are Taylor coefficient multipliers, one obtains
$$R^{\omega,\nu}B_z^{\omega_{+N}}(\xi)=R^{\omega,\nu}R^{\omega,\omega_{+N}}B_z^\omega(\xi)=R^{\omega,\omega_{+N}}B_z^\nu(\xi).$$
The claim now follows by iterating \eqref{plus}.
\end{proof}

It follows that if there exists $N \in \mathbb{N}$ so that $f \in A^1_{\omega_{+N}}$, then
$$R^{\omega,\nu}f(z)=\int_{\mathbb{D}}f(\xi)\overline{B_z^{\nu_{+N}}(\xi)}\omega_{+N}(\xi)dA(\xi).$$
Such is the case, whenever $f$ belongs to any $L^1$ space induced by a standard weight.

We finally point out the obvious identities
\begin{align*}
R^{\omega,\nu}R^{\eta,\sigma}&=R^{\eta,\sigma}R^{\omega,\nu}\\
R^{\omega,\nu}R^{\eta,\omega}&=R^{\eta,\nu}\\
R^{\omega,\nu}R^{\nu,\omega}&=R^{\nu,\omega}R^{\omega,\nu}=I
\end{align*}

It is perhaps a good idea to keep in mind that $\nu$ being "smaller" than $\omega$ means roughly that the operator $R^{\omega,\nu}$ is of derivative type. When the roles are reversed, the operator should be understood as integration.

\section{Bloch space}

Recall that the Bloch space $\mathcal{B}$ consists of analytic functions $f:\mathbb{D}\to \mathbb{C}$ with
$$\|f\|_{\mathcal{B}}=\sup_{z \in \mathbb{D}}(1-|z|^2)|f'(z)|+|f(0)|<\infty.$$

In this section we show that the Bergman projection $P_\omega$ induced by $\omega \in \mathcal{D}$ is surjective from $L^\infty$ onto $\mathcal{B}$. The original proof of the surjectivity belongs to Pel\'aez and R\"atty\"a \cite{PR2018}. However, the proof presented here is simpler, and gives a way to construct several pre-images of a given Bloch function, and therefore of independent interest. Also, the proof of Theorem \ref{bloch} is imporant for understanding the next section of the paper. We remark that the boundedness of this operator is true even under the weaker assumption $\omega \in \widehat{\mathcal{D}}$, and it is easy to deduce from the kernel estimates in \cite{PR2016/1}.

For a radial weight $\omega$, let us denote 
\begin{align*}
\omega_\alpha(r)&=(1-r)^\alpha \omega(r);\\
\widetilde{\omega}(r)&=\frac{\widehat{\omega}(r)}{1-r}; \\
\omega^*(r)&=\int_{r}^1 \omega(s)s\log(s/r)ds.
\end{align*}
The associated weight $\omega^*$ arises from the Littlewood-Paley formula, which is valid for all radial $\omega$ and analytic functions $f,g \in A^2_\omega$.
\begin{equation}\label{lipa}
\int_{\mathbb{D}}f(\xi)\overline{g(\xi)}\omega(\xi)dA(\xi)=4\int_{\mathbb{D}}f'(\xi)\overline{g'(\xi)}\omega^*(\xi)dA(\xi)+\omega(\mathbb{D})f(0)\overline{g(0)}.
\end{equation}
The formula \eqref{lipa} can be found in \cite{PRMEM}, where is is presented in the case $f=g$. By orthogonality, the present formulation follows easily.

It is well-known that for any radial weight $\omega$, one has
$$\omega^*(z)\asymp \widehat{\omega}(z)(1-|z|),\quad |z|\to 1^-,$$
and that $\omega^*$ has a logarithmic singularity at the origin.

Moreover (see, for instance \cite{PPR}), if $\omega \in \mathcal{D}$, then there exist $a,b >0$ so that
$$r\mapsto \frac{\widehat{\omega}(r)}{(1-r)^a}$$ is essentially decreasing, and
$$r\mapsto \frac{\widehat{\omega}(r)}{(1-r)^b}$$ is essentially increasing. A function $\zeta:[0,1)\to [0,\infty)$ is essentially decreasing, if there exists $C>0$ so that for every $0\leq r\leq t<1$, one has $C\zeta(r)\geq \zeta(t)$. Being essentially increasing is defined in a similar manner.

The following properties of $\omega \in \mathcal{D}$ are known, and they can be obtained from the essential monotonicity results above.
\begin{align}
\widehat{\omega_\alpha}(z)&\asymp \widehat{\omega}(z)(1-|z|)^{\alpha},\quad \alpha\geq 0;\\
\widetilde{\omega}(z)&\in \mathcal{R};\\
\widehat{\widehat{\omega}}(z)&\asymp \widehat{\omega}(z)(1-|z|).
\end{align}

Norm estimates for the derivatives of Bergman kernels are crucial in the proof. We note that from the power series representation, it is easily read that for a radial $\omega$, we have
\begin{equation}\label{var}
\partial_{\overline{z}}^n B_z^\omega(\xi)=\left(\frac{\xi}{\overline{z}}\right)^n (B_z^\omega)^{(n)}(\xi).
\end{equation}
It follows easily that 
$$\|\partial_{\overline{z}}^n B_z^\omega\|_{A^p_\nu}\asymp \|(B_z^\omega)^{(n)}\|_{A^p_\nu}, \quad |z|\to 1^-$$ 
for every radial $\nu$ and $p \in (0,\infty)$.

We want to point out that, aside from the formula of an explicit pre-image, the following theorem is not new, as it is proven in \cite{PR2018}. In fact, this result was already mentioned in \cite{PR}. However, there it is noted that obtaining the pre-image can be laborious. The merit of the result below is that it shows that this is not that laborious after all! Pel\'aez and R\"atty\"a have been able to prove even the converse; the class $\mathcal{D}$ is precisely the class of radial weights, for which the theorem below holds.

\begin{theorem}\label{bloch}
Let $\omega \in \mathcal{D}$. Then the Bergman projection $P_\omega:L^\infty \to \mathcal{B}$ is bounded and onto. Moreover, if $\alpha>0$ and $h \in \mathcal{B}$, then
$$g_\alpha(z):= (1-|z|)^\alpha R^{\omega,\omega_\alpha}h(z)$$ belong to $L^\infty$ and $P_\omega(g_\alpha)=h$. 
\end{theorem}

\begin{proof}
The estimate $\|P_\omega(\Phi)\|_{\mathcal{B}}\lesssim \|\Phi\|_\infty$ follows from the kernel estimates in \cite{PR2016/1} and holds even under $\omega \in \widehat{\mathcal{D}}$.
We may assume that $h \in \mathcal{B}$ and $h(0)=0$. Let $\alpha>0$. Then both $h$ and $B_z^{\omega_\alpha}$ belong to $A^2_{\omega_\alpha}$, and we can use the Littlewood-Paley formula to obtain
\begin{align*}
R^{\omega,\omega_\alpha}h(z)&=4\int_{\mathbb{D}}h'(\xi)\overline{(B_z^{\omega_\alpha})'(\xi)}\omega^*(\xi)dA(\xi)\\
&=4\int_{\mathbb{D}}(1-|\xi|)h'(\xi)\overline{(B_z^{\omega_\alpha})'(\xi)}\frac{\omega^*(\xi)}{1-|\xi|}dA(\xi).
\end{align*}
Next, recall that there exists $R \in (0,1)$ so that $\omega^*(\xi)\lesssim \widehat{\omega}(\xi)(1-|\xi|)$ for $\xi\in \mathbb{D}\setminus D(0,R)$. On the other hand, if $\xi \in D(0,R)$, then $|(B_z^{\omega_\alpha})'(\xi)|\leq C(\omega_\alpha,R)$ for every $z \in \mathbb{D}$.

Together with Theorem 1 of \cite{PR2016/1} these observations yield
\begin{align*}
|R^{\omega,\omega_\alpha}h(z)|&\lesssim \|h\|_{\mathcal{B}}\int_{\mathbb{D}}|(B_z^{\omega_\alpha})'(\xi)|\widehat{\omega}(\xi)dA(\xi)\\
&\lesssim \|h\|_{\mathcal{B}} \int_0^{|z|}\frac{\widehat{\widehat{\omega}}(t)dt}{\widehat{\omega}(t)(1-t)^{\alpha}(1-t)^2}\\
&\lesssim \|h\|_{\mathcal{B}} \int_0^{|z|}\frac{dt}{(1-t)^{\alpha+1}}\\
&\lesssim \|h\|_{\mathcal{B}}\frac{1}{(1-|z|)^{\alpha}}, \quad |z|\to 1^-.
\end{align*}
In other words, the function $g_\alpha(z)=(1-|z|)^\alpha R^{\omega,\omega_{\alpha}}h(z)$ belongs to $L^\infty$, so that $R^{\omega,\omega_{\alpha}}h \in A^1_{\omega_\alpha}$. Therefore, by \eqref{int}, we finally note that 
$$P_\omega (g_\alpha)=R^{\omega_\alpha,\omega}R^{\omega,\omega_\alpha}h=h.$$
\end{proof}

The above theorem has a simple and natural corollary. Let $C$ denote the space of complex-valued functions that are continuous in the closed unit disk, and $C_0$ its subspace consisting of functions with zero boundary values. Recall also that the little Bloch space $\mathcal{B}_0$ is the subspace of $\mathcal{B}$ consisting of functions $f$ with
$$\lim_{|z|\to 1^-} (1-|z|^2)|f'(z)|=0.$$
We have the following.

\begin{coro}\label{small}
Let $\omega \in \mathcal{D}$. Then the Bergman projection $P_\omega:X \to \mathcal{B}_0$ is bounded and onto, where $X$ is either $C$ or $C_0$.
\end{coro}

\begin{proof}
Clearly, it suffices to prove the boundedness claim for $X=C$, and the onto part for $X=C_0$. Note first that by orthogonality, if $\Phi$ is a a trigonometric polynomial, then $P_\omega(\Phi)$ is a polynomial; in particular $P_\omega(\Phi)\in \mathcal{B}_0$. The first claim now follows from the norm estimate together with the Stone-Weierstrass theorem.

As for the onto part, we need to show that if $h \in \mathcal{B}_0$ and $\alpha>0$, then the function $g_\alpha(z)=(1-|z|)^\alpha R^{\omega,\omega_{\alpha}}h(z)$ belongs to $C_0$. By the proof of Theorem \ref{bloch}, we know that $h\mapsto g_\alpha$ is bounded $\mathcal{B}\to L^\infty$. If $h$ is a polynomial, then since $R^{\omega,\omega_\alpha}$ is a Taylor coefficient multiplier, also $R^{\omega,\omega_\alpha}h$ is a polynomial. It follows that in this case $g_\alpha \in C_0$. By continuity and completeness, the holds for the closure of polynomials, which is $\mathcal{B}_0$.
\end{proof}

We remark that the above method gives infinitely many pre-images. Such information might be useful for some other problems.

\section{Besov spaces}

Recall that the M\"obius invariant measure is given by 
$$d\lambda(z)=\frac{dA(z)}{(1-|z|^2)^2}.$$

For $1\leq p<\infty$ and $N \in \{2,3,\dots\}$, we define the Besov spaces $\mathcal{B}^p$ to consist of analytic functions $f$ with

$$\|f\|_{\mathcal{B}^p}=\left(\int_{\mathbb{D}}|f^{(N)}(z)|^p (1-|z|)^{Np}d\lambda(z)\right)^{1/p}+\sum_{j=0}^{N-1}|f^{(j)}(0)|<\infty.$$

It is well-known that for fixed $p$, any two choices of $N$ yield the same space, as long as $Np>1$ (choosing $N=1$ would exclude that case $p=1$, so we choose $N\geq 2$).

Denote by $L^p_{\lambda_\omega}$ the spaces of $p$-integrable functions with respect to the measure 
$$d\lambda_\omega(z)=\frac{\omega(z)dA(z)}{\widehat{\omega}(z)(1-|z|)}.$$

The following theorem will be the main result of this section.

\begin{theorem}\label{besov}
Let $\omega \in \mathcal{D}$ and $1\leq p\leq \infty$. Then the Bergman projection $P_\omega:L^p_{\lambda_\omega} \to \mathcal{B}^p$ is bounded and onto.
\end{theorem}

\begin{proof}
We remark that the claimed result is equivalent to obtaining the boundedness of
$$T \Phi(z)=(1-|z|^2)^2\int_{\mathbb{D}}\Phi(\xi)\overline{(\partial_{\overline{z}}^2 B_z^\omega)(\xi)}\omega(\xi)dA(\xi)$$
acting $L^p_{\lambda_\omega}\to L^p_\lambda$. Furthermore, this result has already been established for the case $p=\infty$, which corresponds to $P_\omega:L^\infty \to \mathcal{B}$. Therefore, by the Riesz-Thorin complex interpolation theorem, we will only need to show the boundedness in the case $p=1$. However, by the dualities $(L^1_{\lambda})^* \sim L^\infty$ and $(L^1_{\lambda_\omega})^*\sim L^\infty$, this reduces to showing the boundedness of 
$$T^* \Phi(z)=\widehat{\omega}(z)(1-|z|)\int_{\mathbb{D}}\Phi(\xi)\overline{(B_z^\omega)''(\xi)}dA(\xi).$$
on $L^\infty$. But by Theorem 1 of \cite{PR2016/1}, we have
$$\int_{\mathbb{D}}|(B_z^\omega)''(\xi)|dA(\xi)\asymp \int_0^{|z|}\frac{(1-r)}{\widehat{\omega}(r)(1-r)^{3}}\lesssim \frac{1}{\widehat{\omega}(z)(1-|z|)}.$$

From this estimate, the claim regarding boundedness is immediate.\\

As for the surjectivity. Let $h \in \mathcal{B}^p$ and, without loss of generality, assume that $h(0)=h'(0)=h''(0)=0$. By the proof of Theorem \ref{bloch}, it suffices to show that if $h \in \mathcal{B}^p$, then the function $g_\alpha(z)=(1-|z|^2)^\alpha R^{\omega,\omega_\alpha}h(z)$ belongs to $L^p_{\lambda_\omega}$, whenever $\alpha>1$. By denoting $w^{*2}=(\omega^*)^*$, and using the Littlewood-Paley identity twice, we obtain
$$R^{\omega,\omega_\alpha}h(z)=16\int_{\mathbb{D}}h''(\xi)\overline{(B_z^{\omega_\alpha})''(\xi)}\omega^{*2}(\xi)dA(\xi).$$

Clearly, it will be enough to show the boundedness of $S$ acting $L^p_{\lambda}\to L^p_{\lambda_\omega}$, where
$$
S\Phi(z)=(1-|z|)^\alpha \int_{\mathbb{D}}\Phi(\xi)\overline{(B_z^{\omega_\alpha})''(\xi)}\frac{|\xi|\omega^{*2}(\xi)}{(1-|\xi|)^2}dA(\xi).$$
Here the factor $|\xi|$ can be added, since we assumed $h''(0)=0$. This amounts only to small detail in the proof and is really quite irrelevant.

First, we deal with $S:L^\infty\to L^\infty$. Note that since $\omega \in \mathcal{D}$, we have
$$|\xi|\omega^{*2}(\xi)\lesssim \widehat{\omega}(\xi)(1-|\xi|)^3.$$
Again, Theorem 1 of \cite{PR2016/1} gives
$$|S\Phi(z)|\lesssim \|\Phi\|_{\infty}(1-|z|)^\alpha \int_0^{|z|} \frac{\widehat{\omega}(r)(1-r)^2}{\widehat{\omega}(r)(1-r)^{3+\alpha}}dr\lesssim \|\Phi\|_\infty.$$
Next, we deal with $S:L^1_\lambda\to L^1_{\lambda_\omega}$. By duality, it will be enough to study
$$S^*\Phi(z)=|z|\omega^{*2}(z)\int_{\mathbb{D}}\Phi(\xi)\overline{(\partial_{\overline{z}}^2B_z^{\omega_\alpha})(\xi)}\frac{\omega(\xi)(1-|\xi|)^\alpha}{\widehat{\omega}(\xi)(1-|\xi|)}dA(\xi)$$
acting of $L^\infty$. Note that $\omega^{*2}$ is not bounded, but we are saved by the factor $|z|$ and
$$|z|\omega^{*2}(z)\lesssim \widehat{\omega}(z)(1-|z|)^3.$$
Note that by Lemma 4 in \cite{PPR}, since $\alpha>0$, we have
$$\int_r^1 \frac{\omega(s)(1-s)^{\alpha-1}}{\widehat{\omega}(s)}ds\asymp (1-r)^{\alpha-1}.$$

Another application of Theorem 1 in \cite{PPR} finally gives us
$$|S^*\Phi(z)|\lesssim \|\Phi\|_\infty \widehat{\omega}(z)(1-|z|)^3 \int_0^{|z|}\frac{(1-r)^{\alpha-1}}{\widehat{\omega}(r)(1-r)^{3+\alpha}}dr\lesssim \|\Phi\|_\infty.$$

The proof is completed by an application of the Riesz-Thorin interpolation theorem as done before.
\end{proof}

We point of that if $\omega \in \mathcal{R}$, then $L^p_\lambda=L^p_{\lambda_\omega}$ and the proof is more standard. However, when $\omega \in \mathcal{D}$, it might have singularities and zeroes inside the disk. In general, neither $L^p_\lambda \subset L^p_{\lambda_\omega}$, nor $L^p_{\lambda_\omega}\subset L^p_\lambda$ holds, and for $f \in L^p_\lambda$, the integral
$$P_\omega f(z)=\int_{\mathbb{D}}f(\xi)\overline{B_z^\omega(\xi)}\omega(\xi)dA(\xi)$$

might not be well-defined. Also, to the author's knowledge the proof for surjectivity in \cite{PR2018} does not, in an obvious way, carry over to the Besov space setting.

\section{Final remarks}

We note that if $\omega \in \mathcal{R}$, it is possible to extract the pre-image $h$ of an analytic $f$ under $P_\omega$ by writing
$$h(z)=f(0)+\frac{\widehat{\omega}(z)}{|z|\omega(z)}(2zf'(z)+f(z)-f(0)).$$
If, for instance, $f$ is Bloch function, it is not difficult to see that $h$ will be bounded. However, if $\omega \in \mathcal{D}$, the $h$ above might very well fail to be bounded; $\omega$ might be zero on a set of positive measure, for instance. The author thinks that the fractional derivatives studied here are rather convenient for this purpose.

It is also clear that the norm of $P_\omega : L^\infty \to \mathcal{B}$ equals $$\sup_{z\in \mathbb{D}} \|\partial_{\overline{z}}B_z^\omega\|_{A^1_\omega}.$$
If $\omega=1$, this number equals $8/\pi$, see \cite{Per, Per2}.\\

It is possible to use \eqref{plus} to obtain a fractional Littlewood-Paley identity. By calculating with power series, one sees
\begin{equation}\label{litpal}
\int_{\mathbb{D}}f(z)\overline{g(z)}\omega(z)dA(z)=\int_{\mathbb{D}}R^{\eta,\eta_{+N}}f(z)\overline{R^{\nu,\nu_{+M}}g(z)}\omega_{+N+M}(z)dA(z).
\end{equation}
Here $\omega$ and $\eta$ can be chosen freely, but in practice it might be advantageous to choose them small enough so that the fractional derivatives become integral operators.

Note that by \eqref{lipa}
$$(|\cdot|^2\omega)_n=\int_{\mathbb{D}}|z|^{2n+2}\omega(z)dA(z)=4(n+1)^2\int_{\mathbb{D}}|z|^{2n}\omega^*(z)dA(z)=4(n+1)^2\omega^*_n,$$
so $4\omega^*=(|\cdot|^2 \omega)_{+2}$. In a way the advantage of \eqref{litpal} is that $N+M$ need not be even. Noting that 
$$R^{\omega,\omega_+}f(z)=(zf)'=f(z)+zf'(z),$$
it is possible to easily characterize the Bloch, little Bloch and Besov spaces in terms of these derivatives. Then formula \eqref{litpal} can be used to obtain a perhaps little bit cleaner proof of the Theorems \ref{bloch} and \ref{besov} as well the Corollary \ref{small}. The details are left to the reader. Notice also that \eqref{lipa} is a consequence of \eqref{litpal} by the following reasoning. Given an analytic function $h$, we set $h_0=(h-h(0))/z$, which is also analytic. Notice that $R^{\eta,\eta_+}h_0=h'$. By \eqref{litpal}, we have
$$\int_{\mathbb{D}}f_0(z)\overline{g_0(z)}|z|^2\omega(z)dA(z)=4\int_{\mathbb{D}}f'(z)\overline{g'(z)}\omega^*(z)dA(z).$$
By analyzing the left-hand-side above, we obtain \eqref{lipa}.\\


\end{document}